\documentclass[12pt]{amsart}
\usepackage{mathrsfs}

\usepackage{amssymb,amsfonts,amsthm,amsmath}
\usepackage{epsfig}
\usepackage[utf8]{inputenc}
\usepackage{mathrsfs}
\usepackage{lscape}
\usepackage{amssymb,amsmath,graphicx,color,textcomp, amsthm,bbm,bbold, enumerate,booktabs}
\usepackage{chngcntr}
\usepackage{apptools}
\AtAppendix{\counterwithin{theorem}{section}}
\definecolor{Red}{cmyk}{0,1,1,0}

\definecolor{verde}{cmyk}{1,0,1,0}

\definecolor{loka}{cmyk}{.5,0,1,.5}
\definecolor{azul}{cmyk}{1,1,0,0}

\numberwithin{equation}{section}

\newcommand{\be}{\begin{equation}}
\newcommand{\ee}{\end{equation}}

\newtheorem{definition}{Definition}

\newtheorem{corolario}{Corollary}
\newtheorem{proposicao}{Proposition}
\newtheorem{exemplo}{Example}
\newtheorem{teorema}{Theorem}

\begin{document}
\title{Stability of the fractional Volterra integro-differential equation by means of $\psi-$Hilfer operator}
\author{J. Vanterler da C. Sousa$^1$}
\address{$^1$ Department of Applied Mathematics, Institute of Mathematics,
 Statistics and Scientific Computation, University of Campinas --
UNICAMP, rua S\'ergio Buarque de Holanda 651,
13083--859, Campinas SP, Brazil\newline
e-mail: {\itshape \texttt{vanterlermatematico@hotmail.com, capelas@ime.unicamp.br }}}
\author{E. Capelas de Oliveira$^1$}

\begin{abstract} In this paper, using the Riemann-Liouville fractional integral with respect to another function and the $\psi-$Hilfer fractional derivative, we propose a fractional Volterra integral equation and the fractional Volterra integro-differential equation. In this sense, for this new fractional Volterra integro-differential equation, we study the Ulam-Hyers stability and, also, the fractional Volterra integral equation in the Banach space, by means of the Banach fixed-point theorem. As an application, we present the Ulam-Hyers stability using the $\alpha$-resolvent operator in the Sobolev space $W^{1,1}(\mathbb{R}_{+},\mathbb{C})$.

\end{abstract}
\maketitle

Nowadays one of the areas of great concentration in the scientific community
of mathematical research is the fractional calculus (FC) and its
applications, specifically, in pure mathematics in which some new operators
are presented and discussed in Banach space, and in applied mathematics, in
problems involving memory effects. The concept of a non integer order
derivative is not new, instead, it is as old as the concept of a
integer-order derivative. Since then countless definitions of integrals and
fractional derivatives have appeared, particularly, by replacing the order
of the integer-order derivative with a parameter that may be non-integer.
For example, we mentioned the monographs by Kilbas et al. \cite{KSTJ},
Oldham and Spanier \cite{JERO}, Podlubny \cite{IP}, Samko et al. \cite{SAMKO}
among others. We also mention some recent papers by Katugampola \cite{katu}%
, Almeida \cite{almeida} and Sousa and Oliveira \cite{ZE1,ZE2}.

Fractional differential equations have already proved to be valuable
tools in the modeling of several physical phenomena. The solution of such
fractional differential equations, in general, allows a better evaluation of
the results in several fields of science, for example in engineering,
physics and medicine, among others \cite{JERO,SAMKO,RHM,JLEC1,validation}.
In recent years, many authors have proposed and proved results of existence
and uniqueness for the solutions of such equations, using different methods,
Kilbas et al. \cite{KSTJ}, Podlubny \cite{IP}, Zhou \cite{exis2}, Benchohra
and Lazreg \cite{exis4}, Furati and Kassim \cite{exis5}, Gu and Trujillo 
\cite{exis6}, M\^{a}agli et al. \cite{exis7}, Abbas et al. \cite{exis8},
Yang and Wang \cite{exis9}, Abbas et al. \cite{exis10}, Yong et al. \cite%
{exis11} and so on. The study of fractional differential equations receives
increasing attention and countless manuscripts have been published so far.
Furthermore, it paved way for several new lines of study, such as the 
fractional theory applied to the so-called impulsive equations \cite{imp1,imp2,imp3}.

On the other hand, in 1940, Ulam presented the stability problem of the
solutions of the functional equations regarding the stability of group
homomorphisms. After that, Hyers presented an affirmative answer to the Ulam
question in the context of Banach spaces, which was the first significant
advance in this area. Since then, a large number of articles have been
published in connection with many generalizations of the Ulam problem \cite%
{amann,chicone,hsu}. Also, several investigators used the fractional
derivatives and studied the stability of Ulam-Hyers and Ulam-Hyers-Rassias 
\cite{est4,est6,est8}. Although there are several manuscripts on stability
through fractional derivatives, we believe growth in this area is quite promising.

In this sense, Wang et al. \cite{est3} Benchohra and Lazreg \cite{est5},
Wang and Li \cite{est9} and more recently Sousa and Capelas de Oliveira \cite%
{est1,est2,ser} have investigated stabilities of Ulam-Hyers,
Ulam-Hyers-Rassias, semi-Ulam-Hyers-Rassias, among others, aiming new
results involving different types of stability.

Motivated by the paper of Janfada et al. \cite{janfada}, in this article we
have as main objectives, to study the stability of Ulam-Hyers of the
solution of the fractional Volterra integro-differential equation, Eq.(\ref%
{eq1}) and the solution of the fractional Volterra integral equation, Eq.(%
\ref{eq6}), using Banach fixed-point theorem.

In this paper, we consider the fractional Volterra integro-differential
equation 
\begin{equation}  \label{eq1}
^{H}\mathbb{D}_{0+}^{\alpha ,\beta ;\psi }x\left( t\right) =g\left(
t,x\left(t\right) \right) +\int_{0}^{t}K\left( t,s,x\left( s\right) \right)
ds
\end{equation}
where $^{H}\mathbb{D}_{a+}^{\alpha ,\beta ;\psi }(\cdot)$ is the $\psi-$%
Hilfer fractional derivative \cite{ZE1}, with $0<\alpha <1$, $0\leq \beta
\leq 1$ and $g:\left[ 0,T\right]\times \mathcal{X}\rightarrow \mathcal{X}$, $%
K:[0,T]\times[0,T]\times\mathcal{X}\rightarrow\mathcal{X}$, $x :\left[ 0,T%
\right] \rightarrow \mathcal{X} $ are continuous function and $\mathcal{X}$
is a Banach's space.

The paper is organized as follows: Section 2, presents, as preliminaries,
the definition of the $\psi-$Hilfer fractional derivative, fractional
integral of Riemann-Liouville with respect to another function and some
important results, given as theorems, as well as the spaces in which such
operators and theorems are defined. In Section 3, our main results are stated as
theorems. We first present and discuss the Ulam-Hyers stability for the solution 
of a fractional integro-differential Volterra equation and then we also discuss
the Ulam-Hyers stability for the solution of a fractional integral Volterra
equation, in both cases, using the Banach fixed-point theorem. Section 4,
presents as an example, the Ulam-Hyers stability involving the $\alpha$%
-resolvent operator in the Sobolev space. Concluding remarks close the paper.

\section{Preliminaries}

In this section we introduce the function spaces in which they are of
paramount importance to define the $\psi-$Riemann-Liouville fractional
integral and the $\psi-$Hilfer fractional derivative. In this sense, we
present two important theorems in obtaining the main results. Also, we
introduce the concept of stability of Ulam-Hyers and Ulam-Hyers-Rassias and
Banach fixed point theorem.

Let $\left[ a,b\right] $ $\left( 0<a<b<\infty \right) $ be a finite interval
on the half-axis $\mathbb{R}^{+}$ and let $C\left[ a,b\right] $ be the space
of continuous function $f$ on $\left[ a,b\right] $ with the norm \cite{ZE1} 
\begin{equation*}
\left\Vert f\right\Vert _{C\left[ a,b\right] }=\underset{t\in \left[ a,b%
\right] }{\max }\left\vert f\left( t\right) \right\vert.
\end{equation*}

The weighted space $C_{1-\gamma ;\psi }\left[ a,b\right] $ of continuous $f$
on $\left( a,b\right] $ is defined by \cite{ZE1} 
\begin{equation*}
C_{1-\gamma ;\psi }\left[ a,b\right] =\left\{ f:\left( a,b\right]
\rightarrow \mathbb{R};\left( \psi \left( t\right) -\psi \left( a\right)
\right) ^{1-\gamma }f\left( t\right) \in C\left[ a,b\right] \right\}
\end{equation*}
$0\leq \gamma <1$ with the norm 
\begin{equation*}
\left\Vert f\right\Vert _{C_{1-\gamma ;\psi }\left[ a,b\right] }=\left\Vert
\left( \psi \left( t\right) -\psi \left( a\right) \right) ^{1-\gamma
}f\left( t\right) \right\Vert _{C\left[ a,b\right] }=\underset{t\in \left[
a,b\right] }{\max }\left\vert \left( \psi \left( t\right) -\psi \left(
a\right) \right) ^{1-\gamma }f\left( t\right) \right\vert.
\end{equation*}

The weighted space $C_{\gamma ;\psi }^{n}\left[ a,b\right] $ of continuous $%
f $ on $\left( a,b\right] $ is defined by 
\begin{equation*}
C_{\gamma ;\psi }^{n}\left[ a,b\right] =\left\{ f:\left( a,b\right]
\rightarrow \mathbb{R};f\left( t\right) \in C^{n-1}\left[ a,b\right] ;\text{ 
}f^{\left( n\right) }\left( t\right) \in C_{\gamma ;\psi }\left[ a,b\right]
\right\}
\end{equation*}
$0\leq \gamma <1$ with the norm 
\begin{equation*}
\left\Vert f\right\Vert _{C_{\gamma ;\psi }^{n}\left[ a,b\right] }=\overset{
n-1}{\underset{k=0}{\sum }}\left\Vert f^{\left( k\right) }\right\Vert _{C %
\left[ a,b\right] }+\left\Vert f^{\left( n\right) }\right\Vert _{C_{\gamma
;\psi }\left[ a,b\right] }.
\end{equation*}

\begin{definition}
\textnormal{\cite{ZE1}} Let $\left( a,b\right) $ $\left( -\infty \leq
a<b\leq \infty \right) $ be a finite interval \textrm{(or infinite)} of the
real line $\mathbb{R}$ and let $\alpha >0$. Also let $\psi \left( t\right) $
be an increasing and positive monotone function on $\left( a,b\right] ,$
having a continuous derivative $\psi ^{\prime }\left( t\right)$ \textrm{{(we
denote first derivative as $\dfrac{d}{dt}\psi(t)=\psi^{\prime }(t)$)}} on $%
\left( a,b\right) $. The left-sided fractional integral of a function $f$
with respect to a function $\psi $ on $\left[ a,b\right] $ is defined by 
\begin{equation}  \label{eq2}
I_{a+}^{\alpha ;\psi }f\left( t\right) =\frac{1}{\Gamma \left( \alpha
\right) }\int_{a}^{t}\psi ^{\prime }\left( s\right) \left( \psi \left(
t\right) -\psi \left( s\right) \right) ^{\alpha -1}f\left( s\right) ds.
\end{equation}

The right-sided fractional integral is defined in an analogous form.
\end{definition}

As the aim of this paper is to present some types of the stabilities involving
a class of fractional integro-differential equations by means of $\psi-$%
Hilfer fractional operator, we introduce this fractional operator.

\begin{definition}
\textnormal{\cite{ZE1}} Let $n-1<\alpha <n$ with $n\in \mathbb{N},$ let $\ I=%
\left[ a,b\right] $ be an interval such that $-\infty \leq a<b\leq \infty $
and let $f,\psi \in C^{n}\left[ a,b\right] $ be two functions such that $%
\psi $ is increasing and $\psi ^{\prime }\left( t\right) \neq 0,$ for all $%
t\in I$. The left-sided $\psi-$Hilfer fractional derivative $^{H}\mathbb{D}%
_{a+}^{\alpha ,\beta ;\psi }\left( \cdot \right) $ of a function $f$ of
order $\alpha $ and type $0\leq \beta \leq 1,$ is defined by 
\begin{equation}  \label{eq3}
^{H}\mathbb{D}_{a+}^{\alpha ,\beta ;\psi }f\left( t\right) =I_{a+}^{\beta
\left( n-\alpha \right) ;\psi }\left( \frac{1}{\psi ^{\prime }\left(
t\right) }\frac{d}{dt}\right) ^{n}I_{a+}^{\left( 1-\beta \right) \left(
n-\alpha \right) ;\psi }f\left( t\right) .
\end{equation}

The right-sided $\psi-$Hilfer fractional derivative is defined in an
analogous form.
\end{definition}

If $x(t)$ is a given differentiable function, satisfying 
\begin{equation*}
\left\Vert ^{H}\mathbb{D}_{0+}^{\alpha ,\beta ;\psi }x\left( t\right)
-g\left( t,x\left( t\right) \right) -\int_{0}^{t}K\left( t,s,x\left(
s\right) \right) ds\right\Vert _{C_{1-\gamma ;\psi }\left[ 0,T\right] }\leq
\phi \left( t\right)
\end{equation*}
$\phi \left( t\right) >0$, $t\in \left[ 0,T\right]$, there exists a solution 
$y\left( t\right) $ of Eq.(\ref{eq1}) such that for some $C>0$; 
\begin{equation*}
\left\Vert x\left( t\right) -y\left( t\right) \right\Vert _{C_{1-\gamma
;\psi }\left[ 0,T\right] }\leq C\text{ }\phi \left( t\right)
\end{equation*}
then we say that Eq.(\ref{eq1}) has the Ulam-Hyers stability.

A similar definition can be considered for the fractional Volterra integral
equation 
\begin{equation}  \label{eq6}
x\left( t\right) =g\left( t,x\left( t\right) \right) +\frac{1}{\Gamma \left(
\alpha \right) }\int_{0}^{t}\psi ^{\prime }\left( s\right) \left( \psi
\left( t\right) -\psi \left( s\right) \right) ^{\alpha -1}K\left(
t,s,x\left( s\right) \right) ds.
\end{equation}

\begin{definition}
\label{def3}\textnormal{\cite{fixed}} We say that $d:X\times X\rightarrow %
\left[ 0,\infty \right] $ is a generalized metric on $X$ if:

\begin{enumerate}
\item $d\left( x,y\right) =0$ if and only if $x=y;$

\item $d\left( x,y\right) =d\left( y,x\right) ,$ for all $x,y\in X;$

\item $d\left( x,z\right) \leq d\left( x,y\right) +d\left( y,z\right) $ for
all $x,y,z\in X$.
\end{enumerate}
\end{definition}

\begin{proposicao}
Let $\left( X,\left\Vert \cdot \right\Vert _{C_{1-\gamma ;\psi }}\right) $
be a normed space. Then the function $d:X\times X\rightarrow \mathbb{R} $, $%
d\left( x,y\right) =\left\Vert x-y\right\Vert _{C_{1-\gamma ;\psi }}$ is a
metric on $X$. It is called the metric induced from the norm of the weighted
space.
\end{proposicao}

\begin{proof} The axioms 1 and 2 of Definition \ref{def3} are clear. If $x,y,z\in X$, then 
\begin{eqnarray*}
d\left( x,z\right)  &=&\left\Vert x-z\right\Vert _{C_{1-\gamma ;\psi }\left[
0,T\right] }=\left\Vert \left( \psi \left( t\right) -\psi \left( a\right)
\right) ^{1-\gamma }\left( x\left( t\right) -z\left( t\right) \right)
\right\Vert _{C\left[ 0,T\right] } \\
&\leq &\left\Vert \left( \psi \left( t\right) -\psi \left( a\right) \right)
^{1-\gamma }\left( x\left( t\right) -y\left( t\right) \right) \right\Vert _{C
\left[ 0,T\right] } \\
&&+\left\Vert \left( \psi \left( t\right) -\psi \left( a\right) \right)
^{1-\gamma }\left( y\left( t\right) -z\left( t\right) \right) \right\Vert _{C
\left[ 0,T\right] } \\
&=&d\left( x,y\right) +d\left( y,z\right) .
\end{eqnarray*}
and so the triangle inequality follows.
\end{proof}

\begin{teorema}
\label{teo1} Let $\left( X,d\right) $ be a generalized complete metric
space. Assume that $\Lambda :X\rightarrow X$ is a strictly contractive
operator with the Lipschitz constant $L<1$. If there exists a nonnegative
integer $k$ such that $d\left( \Lambda ^{k+1}x,\Lambda ^{k}x\right) <1,$ for
some $x\in X,$ then the following are true:

\begin{enumerate}
\item The sequence $\left\{ \Lambda ^{n}x\right\} $ converges to a fixed
point $x^{\ast }$ of $\Lambda ;$

\item $x^{\ast }$ is the unique fixed point of $\Lambda $ in 
\begin{equation*}
X^{\ast }=\left\{ y\in X;\text{ }d\left( \Lambda ^{k}x,y\right)
<\infty\right\}.
\end{equation*}

\item If $y\in X^{\ast }$, then $d\left( y,x^{\ast }\right) \leq \frac{1}{1-L%
}d\left( \Lambda y,y\right)$.
\end{enumerate}
\end{teorema}

\begin{proof}
See \textnormal{\cite{ZE1}}.
\end{proof}

\begin{teorema}
\label{teo2} Let $f\in C^{1}\left[ a,b\right]$, $\alpha >0$ and $0\leq \beta
\leq 1$, we have 
\begin{equation*}
^{H}\mathbb{D}_{0+}^{\alpha ,\beta ;\psi }\text{ }I_{0+}^{\alpha ;\psi
}f\left( t\right) =f\left( t\right) .
\end{equation*}
\end{teorema}

\begin{proof}
See \cite{fixed}.
\end{proof}

\begin{corolario}
\label{cor1}\textnormal{\cite{engel}} Let $A$ generate an analytic $\alpha -$%
resolvent $\left\{U_{\alpha }\left( t\right) \right\} _{t\geq 0},$ $\alpha >0
$ on a Banach space $\mathcal{X}$ and let $0<\mu <1$. If one takes $x\in 
\mathcal{X} _{\mu +1}$ and $f\in C_{0}\left( \mathbb{R} _{+}, \mathcal{X}
_{\mu }\right) ,$ then \textnormal{\ Eq.(\ref{eq1})} has a unique solution $%
u\in C_{0}^{1}\left( \mathbb{R} _{+}, \mathcal{X}_{\mu }\right) \cap C\left( 
\mathbb{R} _{+}, \mathcal{X}_{\mu +1}\right)$.
\end{corolario}

The proof is obtained by taking as solution $u$ the first coordinate of $%
t\rightarrow I_{-1}\left( t\right) \left( 
\begin{array}{c}
x \\ 
f%
\end{array}
\right) $ for $\left( 
\begin{array}{c}
x \\ 
f%
\end{array}
\right) \in D\left( \overline{A}\right) ,$ where $I_{-1}\left( t\right)
:=\left( 
\begin{array}{cc}
T\left( t\right) & R_{-1}\left( t\right) \\ 
0 & S_{-1}\left( t\right)%
\end{array}
\right) $ $t\geq 0$ and $D\left( \overline{A}\right) :=\mathcal{X}_{\mu
+1}\times C_{0}\left( \mathbb{R} _{+},\mathcal{X}_{\mu }\right)$.

In this paper, using the Theorem \ref{teo1}, we shall study the Ulam-Hyers
stability of Eq.(\ref{eq1}) and Eq.(\ref{eq6}). Next, some examples of these
equation and their Ulam-Hyers stability will be considered.


\section{Hyers-Ulam stability}

In this section, we will study the Ulam-Hyers stability for the fractional
Volterra integral equation and fractional Volterra integro-differential
equation by means of $\psi-$Hilfer fractional derivative and the fractional
integral of a function with respect to another function $\psi(\cdot)$, in
the Banach's space.

\begin{teorema}
Suppose $\mathcal{X}$ a Banach's space and $L, L_{1}, L_{2}$ and $T$ are
positive constants for which $0<L_{1}+\left( L_{1}+L_{2}+L_{2}L\right) L<1$.
Let $g:\left[ 0,T\right] \times \mathcal{X} \rightarrow \mathcal{X}$, $K:%
\left[ 0,T\right] \times \left[ 0,T\right] \times \mathcal{X} \rightarrow 
\mathcal{X} $ and $\phi :\left[ 0, T \right] \rightarrow \left( 0,\infty
\right) $ be continuous and satisfying 
\begin{equation}  \label{eq8}
\left\Vert g\left( t,x\right) -g\left( t,y\right) \right\Vert _{C_{1-\gamma
;\psi }}\leq L_{1}\left\Vert x-y\right\Vert _{C_{1-\gamma ;\psi }\left[ 0,T%
\right]},
\end{equation}
\begin{equation}  \label{eq9}
\left\Vert K\left( t,s,x\left( s\right) \right) -K\left( t,s,y\left(
s\right) \right) \right\Vert _{C_{1-\gamma ;\psi }}\leq L_{2}\left\Vert
x-y\right\Vert _{C_{1-\gamma ;\psi }\left[ 0,T\right] }
\end{equation}
and 
\begin{equation}  \label{eq10}
\frac{1}{\Gamma \left( \alpha \right) }\int_{0}^{t}\psi ^{\prime }\left(
s\right) \left( \psi \left( t\right) -\psi \left( s\right) \right) ^{\alpha
-1}\phi \left( s\right) ds\leq L\phi \left( t\right)
\end{equation}
for all $s,t\in \left[ 0,T\right] $ and $x,y\in \mathcal{X}$. If $f:\left[
0,T\right] \rightarrow \mathcal{X} $ is a differentiable function satisfying 
\begin{equation}  \label{eq11}
\left\Vert ^{H}\mathbb{D}_{0+}^{\alpha ,\beta ;\psi }f\left( t\right)
-g\left(t,f\left( t\right) \right) -\int_{0}^{t}K\left( t,s,f\left( s\right)
\right)ds\right\Vert _{C_{1-\gamma ;\psi }\left[ 0,T\right] }\leq \phi
\left(t\right)
\end{equation}
$t\in \lbrack 0,T],$ then there exists a unique differentiable function $%
f_{0}:\left[ 0,T\right] \rightarrow \mathcal{X} $ such that for each $t\in %
\left[0,T\right] $ 
\begin{equation}
^{H}\mathbb{D}_{0+}^{\alpha ,\beta ;\psi }f_{0}\left( t\right) =g\left(
t,f_{0}\left( t\right) \right) +\int_{0}^{t}K\left( t,s,f_{0}\left( s\right)
\right) ds  \label{eq12}
\end{equation}%
and 
\begin{eqnarray}
&&\left\Vert ^{H}\mathbb{D}_{0+}^{\alpha ,\beta ;\psi }f\left( t\right) -%
\text{ }^{H}\mathbb{D}_{0+}^{\alpha ,\beta ;\psi }f_{0}\left( t\right)
\right\Vert _{C_{1-\gamma ;\psi }}+\left\Vert f\left( t\right) -f_{0}\left(
t\right) \right\Vert _{C_{1-\gamma ;\psi }}  \notag  \label{eq13} \\
&\leq &\frac{1+L}{1-L_{1}+\left( L_{1}+L_{2}+L_{2}L\right) L}\phi \left(
t\right) .
\end{eqnarray}
\end{teorema}

\begin{proof}Consider the following set $M=\left\{ x:\left[ 0,T\right] \rightarrow \mathcal{X}, \text{ }x\text{ is differentiable}\right\} $ and define a mapping $d:M\times M\rightarrow \left[ 0,\infty \right] $ by
\begin{equation*}
d\left( x,y\right) =\inf \left\{ 
\begin{array}{c}
C\in \left[ 0,\infty \right] :\left\Vert ^{H}\mathbb{D}_{0+}^{\alpha ,\beta
;\psi }x\left( t\right) -\text{ }^{H}\mathbb{D}_{0+}^{\alpha ,\beta ;\psi
}y\left( t\right) \right\Vert _{C_{1-\gamma ;\psi }} \\ 
+\left\Vert x\left( t\right) -y\left( t\right) \right\Vert _{C_{1-\gamma
;\psi }}\leq C\phi \left( t\right) ,\text{ }t\in \left[ 0,T\right] .%
\end{array}%
\right\}  
\end{equation*}

We show that $\left( M,d\right) $ is a complete generalized metric spaces. It suffices to prove the triangle inequality and the completeness of this space. Assume that $d\left( x,y\right) >d\left( x,z\right) +d\left( z,y\right)$,
for some $x,y\in M$. Then, there exists $t_{0}\in \left[ 0,T\right] $ with
\begin{equation*}
\left\Vert ^{H}\mathbb{D}_{0+}^{\alpha ,\beta ;\psi }x\left( t_{0}\right) -\text{ }^{H}\mathbb{D}_{0+}^{\alpha ,\beta ;\psi }y\left( t_{0}\right) \right\Vert_{C_{1-\gamma ;\psi }}+\left\Vert x\left( t_{0}\right) -y\left( t_{0}\right) \right\Vert _{C_{1-\gamma ;\psi }}>\left( d\left( x,z\right) +d\left(z,y\right) \right) \phi \left( t_{0}\right) .
\end{equation*}

Thus, by definition of $d$, we get
\begin{eqnarray*}
&&\left\Vert ^{H}\mathbb{D}_{0+}^{\alpha ,\beta ;\psi }x\left( t_{0}\right) -\text{ }^{H}\mathbb{D}_{0+}^{\alpha ,\beta ;\psi }y\left( t_{0}\right) \right\Vert
_{C_{1-\gamma ;\psi }}+\left\Vert x\left( t_{0}\right) -y\left( t_{0}\right)
\right\Vert _{C_{1-\gamma ;\psi }}  \notag \\
&>&\left\Vert ^{H}\mathbb{D}_{0+}^{\alpha ,\beta ;\psi }x\left( t_{0}\right) -\text{ }^{H}D_{0+}^{\alpha ,\beta ;\psi }z\left( t_{0}\right) \right\Vert
_{C_{1-\gamma ;\psi }}+\left\Vert x\left( t_{0}\right) -z\left( t_{0}\right)
\right\Vert _{C_{1-\gamma ;\psi }}  \notag \\
&&+\left\Vert ^{H}\mathbb{D}_{0+}^{\alpha ,\beta ;\psi }z\left( t_{0}\right) -\text{ }^{H}\mathbb{D}_{0+}^{\alpha ,\beta ;\psi }y\left( t_{0}\right) \right\Vert
_{C_{1-\gamma ;\psi }}+\left\Vert z\left( t_{0}\right) -y\left( t_{0}\right)
\right\Vert _{C_{1-\gamma ;\psi }}
\end{eqnarray*}
which is a contradiction. Now, we show that $\left( M,d\right) $ is complete.

Let $\left\{ x_{n}\right\} $ be a Cauchy sequence in $\left( M,d\right) $. This, by definition of $d$, implies that $\forall \varepsilon >0,$ $\exists
N_{\varepsilon }\in \mathbb{N}$, $\forall m,n\geq N_{\varepsilon },$ $\forall t\in \left[ 0,T\right]$,
\begin{equation}\label{eq17}
\left\Vert ^{H}\mathbb{D}_{0+}^{\alpha ,\beta ;\psi }x_{n}\left( t\right) -\text{ }^{H}\mathbb{D}_{0+}^{\alpha ,\beta ;\psi }x_{m}\left( t\right) \right\Vert_{C_{1-\gamma ;\psi }}+\left\Vert x_{n}\left( t\right) -x_{m}\left( t\right)\right\Vert _{C_{1-\gamma ;\psi }}<\varepsilon \phi \left( t\right).
\end{equation}

Now, by continuity of $\phi $ on compact interval $\left[ 0,T\right]$, Eq.(\ref{eq17}) implies that $\left\{ x_{n}\right\} $ and $\left\{ ^{H}\mathbb{D}_{0+}^{\alpha ,\beta; \psi }x_{n}\right\} $ are uniformly convergent on $\left[ 0,T\right]$. So there exists a differentiable function $x$ such that $\left\{ x_{n}\right\} $ and $\left\{ ^{H}\mathbb{D}_{0+}^{\alpha, \beta; \psi }x_{n}\right\} $ are uniformly convergent to $x$ and $^{H}\mathbb{D}_{0+}^{\alpha, \beta; \psi }x$, respectively.

Hence $x\in M$ and from Eq.(\ref{eq17}), letting $m\rightarrow \infty $, we have $ \forall \varepsilon >0$, $\exists N_{\varepsilon }\in \mathbb{N}$, $\forall n\geq N_{\varepsilon }$, $\forall t\in \left[ 0,T\right] $
\begin{equation*}
\left\Vert ^{H}\mathbb{D}_{0+}^{\alpha ,\beta ;\psi }x_{n}\left( t\right) -\text{ }^{H}\mathbb{D}_{0+}^{\alpha ,\beta ;\psi }x\left( t\right) \right\Vert _{C_{1-\gamma;\psi }}+\left\Vert x_{n}\left( t\right) -x\left( t\right) \right\Vert_{C_{1-\gamma ;\psi }}<\varepsilon \phi \left( t\right) .
\end{equation*}

Consequently, $\forall \varepsilon >0,$ $\exists N_{\varepsilon }\in  \mathbb{N}$, $\forall n\geq N_{\varepsilon }$, $d\left( x_{n},x\right) \leq \varepsilon $ and so $\left( M,d\right) $ is complete. Now define, $\Lambda
:M\rightarrow M$, by
\begin{equation*}
\Lambda \left( x\left( t\right) \right) =I_{0+}^{\alpha ;\psi }g\left( x\left( t\right) \right) +I_{0+}^{\alpha ;\psi }\left[ \int_{0}^{t}K\left( \tau ,s,x\left( s\right) \right) ds\right].
\end{equation*}

First, we show that $\Lambda $ is strictly contractive. Suppose $x,y\in \mathbb{N}$, $C_{xy}\in \left[ 0,\infty \right] $ and $d\left( x,y\right) \leq C_{xy}$. Thus, for all $t\in \left[ 0,T\right]$,
\begin{equation*}
\left\Vert ^{H}\mathbb{D}_{0+}^{\alpha ,\beta ;\psi }x\left( t\right) -\text{ }^{H}\mathbb{D}_{0+}^{\alpha ,\beta ;\psi }y\left( t\right) \right\Vert _{C_{1-\gamma ;\psi }}+\left\Vert x\left( t\right) -y\left( t\right) \right\Vert_{C_{1-\gamma ;\psi }}<C_{xy}\text{ }\phi \left( t\right) .
\end{equation*}

Hence, by Eq.(\ref{eq8}), Eq.(\ref{eq9}) and Eq.(\ref{eq10}), we get
\begin{eqnarray*}
&&\left\Vert ^{H}\mathbb{D}_{0+}^{\alpha ,\beta ;\psi }\left( \Lambda
x\left( t\right) -\Lambda y\left( t\right) \right) \right\Vert _{C_{1-\gamma
;\psi }}+\left\Vert \Lambda x\left( t\right) -\Lambda y\left( t\right)
\right\Vert _{C_{1-\gamma ;\psi }} \\
&\leq &\left\Vert g\left( t,x\left( t\right) \right) -g\left( t,y\left(
t\right) \right) \right\Vert _{C_{1-\gamma ;\psi }}+\int_{0}^{t}\left\Vert
K\left( t,s,x\left( s\right) \right) -K\left( t,s,y\left( s\right) \right)
\right\Vert _{C_{1-\gamma ;\psi }}ds \\
&&+\left\Vert I_{0+}^{\alpha ;\psi }\left( g\left( t,x\left( t\right)
\right) -g\left( t,y\left( t\right) \right) \right) \right\Vert
_{C_{1-\gamma ;\psi }} \\
&&+\left\Vert I_{0+}^{\alpha ;\psi }\left[ \int_{0}^{t}\left( K\left( \tau
,s,x\left( s\right) \right) -K\left( \tau ,s,y\left( s\right) \right)
\right) \right] ds\right\Vert _{C_{1-\gamma ;\psi }} \\
&\leq &L_{1}\left\Vert x\left( t\right) -y\left( t\right) \right\Vert
_{C_{1-\gamma ;\psi }}+L_{2}\int_{0}^{t}\left\Vert x\left( \tau \right)
-y\left( \tau \right) \right\Vert _{C_{1-\gamma ;\psi }}d\tau  \\
&&+I_{0+}^{\alpha ;\psi }\left( L_{1}\left\Vert x\left( t\right) -y\left(
t\right) \right\Vert _{C_{1-\gamma ;\psi }}\right)+I_{0+}^{\alpha ;\psi }\left[ \int_{0}^{t}L_{2}\left\Vert x\left( t\right) -y\left( t\right) \right\Vert _{C_{1-\gamma ;\psi }}ds\right]  \\
&\leq &L_{1}C_{xy}\phi \left( t\right) +L_{2}C_{xy}\int_{0}^{t}\phi \left(
\tau \right) d\tau +L_{1}C_{xy}I_{0+}^{\alpha ;\psi }\phi \left( t\right)
+L_{2}I_{0+}^{\alpha ;\psi }\left[ C_{xy}\int_{0}^{t}\phi \left( s\right) ds%
\right]  \\
&\leq &\left[ L_{1}+\left( L_{2}+L_{1}+L_{2}L\right) L\right] C_{xy}\phi
\left( t\right) .
\end{eqnarray*}

Thus, implies that
\begin{equation}\label{eq21}
d\left( \Lambda x,\Lambda y\right) \leq \left[ L_{1}+\left( L_{2}+L_{1}+L_{2}L\right) L\right] d\left( x,y\right) .
\end{equation}

So $\Lambda $ is strictly contractive, since $0<L_{1}+\left( L_{2}+L_{1}+L_{2}L\right) L<1$. On the other hand, trivially $f\in M$ and by Eq.(\ref{eq11}), we get
\begin{eqnarray*}
&&\left\Vert ^{H}\mathbb{D}_{0+}^{\alpha ,\beta ;\psi }\left( \Lambda f\left(t\right) -f\left( t\right) \right) \right\Vert _{C_{1-\gamma ;\psi
}}+\left\Vert \Lambda f\left( t\right) -f\left( t\right) \right\Vert
_{C_{1-\gamma ;\psi }}  \notag \\
&=&\left\Vert f\left( t,x\left( t\right) \right) +\int_{0}^{t}K\left( t
,s,x\left( s\right) \right) ds-\text{ }^{H}\mathbb{D}_{0+}^{\alpha ,\beta ;\psi }f\left( t\right) \text{ }\right\Vert _{C_{1-\gamma ;\psi }} \notag \\
&&+\left\Vert I_{0+}^{\alpha ;\psi }f\left( t,x\left( t\right) \right)
+I_{0+}^{\alpha ;\psi }\left[ \int_{0}^{t}K\left( \tau ,s,x\left( s\right)
\right) ds\right] -f\left( t\right) \right\Vert _{C_{1-\gamma ;\psi }} 
\notag \\
&\leq &\phi \left( t\right) +\left\Vert I_{0+}^{\alpha ;\psi }f\left(
t,x\left( t\right) \right) +I_{0+}^{\alpha ;\psi }\left[ \int_{0}^{t}K\left(
\tau ,s,x\left( s\right) \right) ds\right] -f\left( t\right) \right\Vert
_{C_{1-\gamma ;\psi }}  \notag \\
&\leq &\left( 1+L\right) \phi \left( t\right) .
\end{eqnarray*}

Consequently,
\begin{equation}\label{eq22}
d\left( \Lambda f,f\right) \leq 1+L<\infty ,\text{ }L<1.
\end{equation}

By means of the item 2 of Theorem \ref{teo1}, there exists a unique element $f_{0}\in M^{\ast }=\left\{ y\in M:d\left( \Lambda f,y\right) <\infty \right\} $ such that $\Lambda f_{0}=f_{0}$ or equivalent
\begin{equation}\label{eq23}
f_{0}\left( t\right) =I_{0+}^{\alpha ;\psi }\left( g\left( t,f_{0}\left( t\right) \right) \right) +I_{0+}^{\alpha ;\psi }\left[ \int_{0}^{t}K\left( \tau ,s,f_{0}\left( s\right) \right) ds\right].
\end{equation}

Note that, $f_{0}$ is differentiable and $g,k$ are continuous, then applying the $\psi-$Hilfer fractional derivative $^{H}\mathbb{D}_{0+}^{\alpha ,\beta ;\psi }(\cdot)$ on both sides of Eq.(\ref{eq23}) and by means of Theorem \ref{teo2}, we have 
\begin{equation*}
^{H}\mathbb{D}_{0+}^{\alpha ,\beta ;\psi }f_{0}\left( t\right) =g\left( t,f_{0}\left(t\right) \right) +\int_{0}^{t}K\left( t,s,f_{0}\left( s\right) \right)ds.
\end{equation*}

Also, from item 2 of Theorem \ref{teo1} and Eq.(\ref{eq22}), we have
\begin{eqnarray*}
d\left( f,f_{0}\right) &\leq &\frac{1}{1-\left[ L_{1}+\left( L_{2}+L_{1}+L_{2}L\right) L\right] }d\left( \Lambda f,f\right)  \notag \\ &\leq &\frac{1+L}{1-\left[ L_{1}+\left( L_{2}+L_{1}+L_{2}L\right) L\right] }.
\end{eqnarray*}

In view of definition of $d$ we can conclude that the inequality Eq.(\ref{eq13}) holds, for all $t\in \left[ 0,T\right]$. First, we consider
\begin{equation*}
\xi =\frac{1+L}{1-\left[ L_{1}+\left( L_{2}+L_{1}+L_{2}L\right) L\right] }.
\end{equation*}

Let $h$ be another differentiable function satisfying Eq.(\ref{eq12}) and Eq.(\ref{eq13}). Then $f\in M,$ $d\left( f,h\right) <\xi$, and
\begin{equation}\label{eq27}
^{H}\mathbb{D}_{0+}^{\alpha ,\beta ;\psi }h\left( t\right) =g\left( t,h\left(t\right) \right) +\int_{0}^{t}K\left( t,s,h\left( s\right) \right) ds.
\end{equation}

Thus, we prove the uniqueness of $f_{0}$. To this end it is enough to show that $h$ is a fixed point of $\Lambda $ and $h\in M^{\ast }$. Using Eq.(\ref{eq27}), on can see that $\Lambda h=h$. We show that $d\left( \Lambda f,h\right) <\infty$. From Eq.(\ref{eq27}) and the fact $d\left( f,h\right) <\xi$, we get
\begin{eqnarray*}
&&\left\Vert ^{H}\mathbb{D}_{0+}^{\alpha ,\beta ;\psi }\left( \Lambda
f\left( t\right) -h\left( t\right) \right) \right\Vert _{C_{1-\gamma ;\psi
}}+\left\Vert \Lambda f\left( t\right) -h\left( t\right) \right\Vert
_{C_{1-\gamma ;\psi }}  \notag \\
&=&\left\Vert g\left( t,f\left( t\right) \right) +\int_{0}^{t}K\left(
t,s,f\left( s\right) \right) ds-g\left( t,h\left( t\right) \right)
-\int_{0}^{t}K\left( t,s,h\left( s\right) \right) ds\right\Vert
_{C_{1-\gamma ;\psi }}  \notag \\
&&+\left\Vert 
\begin{array}{c}
I_{0+}^{\alpha ;\psi }g\left( t,f\left( t\right) \right) +I_{0+}^{\alpha
;\psi }\left[ \displaystyle\int_{0}^{t}K\left( \tau,s,f\left( s\right) \right) ds\right]  \\ 
-I_{0+}^{\alpha ;\psi }g\left( t,h\left( t\right) \right) -I_{0+}^{\alpha
;\psi }\left[ \displaystyle\int_{0}^{t}K\left( \tau,s,h\left( s\right) \right) ds\right] 
\end{array}%
\right\Vert _{C_{1-\gamma ;\psi }}  \notag \\
&\leq &\left\Vert g\left( t,f\left( t\right) \right) -g\left( t,h\left(
t\right) \right) \right\Vert _{C_{1-\gamma ;\psi }}+\int_{0}^{t}\left\Vert
K\left( t,s,f\left( s\right) \right) -K\left( t,s,h\left( s\right) \right)
\right\Vert _{C_{1-\gamma ;\psi }}ds  \notag \\
&&+\left\Vert I_{0+}^{\alpha ;\psi }\left[ g\left( t,f\left( t\right)
\right) -g\left( t,h\left( t\right) \right) \right] \right\Vert
_{C_{1-\gamma ;\psi }}  \notag \\
&&+\left\Vert I_{0+}^{\alpha ;\psi }\left[ \int_{0}^{t}\left[ K\left(
\tau,s,f\left( s\right) \right) -K\left( \tau,s,h\left( s\right) \right) \right] ds%
\right] \right\Vert _{C_{1-\gamma ;\psi }}  \notag \\
&\leq &L_{1}\left\Vert f\left( t\right) -h\left( t\right) \right\Vert
_{C_{1-\gamma ;\psi }}+L_{2}\int_{0}^{t}\left\Vert f\left( s\right) -h\left(
s\right) \right\Vert _{C_{1-\gamma ;\psi }}ds  \notag \\
&&+L_{1}\text{ }I_{0+}^{\alpha ;\psi }\left( \left\Vert f\left( t\right)
-h\left( t\right) \right\Vert _{C_{1-\gamma ;\psi }}\right) +L_{2}\text{ }%
I_{0+}^{\alpha ;\psi }\left( \int_{0}^{t}\left\Vert f\left( s\right)
-h\left( s\right) \right\Vert _{C_{1-\gamma ;\psi }}ds\right)   \notag \\
&\leq &\left[ L_{1}+\left( L_{2}+L_{1}+L_{2}L\right) L\right] \xi 
\end{eqnarray*}
which implies that $d\left( \Lambda f,h\right) \leq \left[ L_{1}+\left(
L_{2}+L_{1}+L_{2}L\right) L\right] \xi <\infty$.
\end{proof}

The next step of the main result of this article, is the proof of the
stability of the fractional Volterra integral equation.

\begin{teorema}
Suppose $\mathcal{X}$ a Banach's space and $L,L_{1},L_{2\text{ }}$ and $T$
are positive constants for which $0<\left( L_{1}+L_{2}\right) L<1$. Let $g:%
\left[ 0,T\right] \times \mathcal{X} \rightarrow \mathcal{X}$, $K:\left[ 0,T%
\right] \times \left[ 0,T\right] \times \mathcal{X} \rightarrow \mathcal{X} $
and $\phi :\left[ 0,T\right] \rightarrow \left( 0,\infty \right) $ be
continuous functions satisfying 
\begin{equation}  \label{eq88}
\left\Vert g\left( t,x\right) -g\left( t,y\right) \right\Vert _{C_{1-\gamma
;\psi }}\leq L_{1}\left\Vert x-y\right\Vert _{C_{1-\gamma ;\psi }},
\end{equation}
\begin{equation}  \label{eq29}
\left\Vert K\left( t,s,x\right) -K\left( t,s,y\right) \right\Vert
_{C_{1-\gamma ;\psi }}\leq L_{2}\left\Vert x-y\right\Vert _{C_{1-\gamma
;\psi }}
\end{equation}
and 
\begin{equation}  \label{eq90}
\frac{1}{\Gamma \left( \alpha \right) }\int_{0}^{t}\psi ^{\prime }\left(
s\right) \left( \psi \left( t\right) -\psi \left( s\right) \right) ^{\alpha
-1}\phi \left( s\right) ds\leq L\phi \left( t\right) ,
\end{equation}
for all $s,t\in \left[ 0,T\right] $ and $x,y\in \mathcal{X} $. If $f:\left[
0,T\right] \rightarrow \mathcal{X} $ is a continuous function satisfying 
\begin{equation}  \label{eq30}
\left\Vert f\left( t\right) -g\left( t,f\left( t\right) \right) -\frac{1}{%
\Gamma \left( \alpha \right) }\int_{0}^{t}\psi ^{\prime }\left( s\right)
\left( \psi \left( t\right) -\psi \left( s\right) \right) ^{\alpha
-1}K\left( t,s,f\left( s\right) \right) ds\right\Vert _{C_{1-\gamma ;\psi
}}\leq \phi \left( t\right),
\end{equation}
$\text{ }t\in \left[ 0,T\right]$, then there exists a unique continuous
function $f_{0}:\left[ 0,T\right] \rightarrow \mathcal{X} $ such that 
\begin{equation}  \label{eq31}
f_{0}\left( t\right) =g\left( t,f_{0}\left( t\right) \right) +\frac{1}{%
\Gamma \left( \alpha \right) }\int_{0}^{t}\psi ^{\prime }\left( s\right)
\left( \psi \left( t\right) -\psi \left( s\right) \right) ^{\alpha-1}K\left(
t,s,f\left( s\right) \right) ds
\end{equation}
and 
\begin{equation}  \label{a32}
\left\Vert f\left( t\right) -f_{0}\left( t\right) \right\Vert _{C_{1-\gamma
;\psi }}\leq \frac{1}{1-\left( L_{1}+L_{2}\right) L}\phi \left( t\right).
\end{equation}
\end{teorema}

\begin{proof} Consider the following set
\begin{equation*}
M:=\left\{ x:\left[ 0,T\right] \rightarrow \mathcal{X} :\text{ }x\text{ is continuous}\right\} 
\end{equation*}
and define $d:M\times M\rightarrow \left[ 0,\infty \right] $ by\qquad 
\begin{equation*}
d\left( x,y\right) =\inf \left\{ C\in \left[ 0,\infty \right] :\left\Vert x\left( t\right) -y\left( t\right) \right\Vert _{C_{1-\gamma ;\psi }}\leq C\phi \left( t\right) ,\text{ }t\in \left[ 0,T\right] \right\} .
\end{equation*}

With a similar argument to the proof of Theorem 2, one can see that $\left( M,d\right) $ is a complete generalized metric space. Now define $\Lambda $ on $M$ as in the proof of Theorem \ref{teo2}, one can verify that for, any $x,y\in M$,
\begin{equation*}
d\left( \Lambda x,\Lambda y\right) \leq \left( L_{1}+L_{2}\right) L\text{ }d\left( x,y\right).
\end{equation*}

The fact that $0<\left( L_{1}+L_{2}\right) L<1$ implies that $\Lambda $ is strictly contractive. Also from Eq.(\ref{eq30}), we obtain $d\left( \Lambda f,f\right) \leq 1<\infty $ and so by Theorem \ref{teo1}, $\Lambda $ has a unique fixed point $f_{0\text{ }}$ in the set $M^{\ast }=\left\{ y\in M:d\left( \Lambda f,y\right) <\infty \right\} .$ Let $h$ be another continuous function satisfying Eq.(\ref{eq31}) and Eq.(\ref{a32}). Thus $ f\in M$, 
$$d\left(f,h\right) <\frac{1}{1-\left( L_{1}+L_{2}\right) L}$$
and
\begin{equation*}
h\left( t\right) =g\left( t,h\left( t\right) \right) +\frac{1}{\Gamma \left( \alpha \right) }\int_{0}^{t}\psi ^{\prime }\left( s\right) \left( \psi \left( t\right) -\psi \left( s\right) \right) ^{\alpha -1}K\left( t,s,h\left( s\right) \right) ds.
\end{equation*}

For the uniqueness of $h,$ it is enough to show that $h$ is a fixed point of $\Lambda $ and $h\in M^{\ast }$. Using Eq.(\ref{eq31}), we have $\Lambda h=h$. Also from Eq.(\ref{eq31}), the fact that
\begin{equation*}
d\left( f,h\right) \leq \frac{1}{1-\left( L_{1}+L_{2}\right) L},
\end{equation*}
by Eq.(\ref{eq88}), Eq.(\ref{eq29}) and Eq.(\ref{eq90}), we have
\begin{eqnarray*}
&&\left\Vert \Lambda f\left( t\right) -h\left( t\right) \right\Vert
_{C_{1-\gamma ;\psi }} \\
&=&\left\Vert 
\begin{array}{c}
g\left( t,f\left( t\right) \right) -\dfrac{1}{\Gamma \left( \alpha \right) }%
\displaystyle\int_{0}^{t}\psi ^{\prime }\left( s\right) \left( \psi \left(
t\right) -\psi \left( s\right) \right) ^{\alpha -1}K\left( t,s,h\left(
s\right) \right) ds \\ 
-g\left( t,h\left( t\right) \right) +\dfrac{1}{\Gamma \left( \alpha \right) }%
\displaystyle\int_{0}^{t}\psi ^{\prime }\left( s\right) \left( \psi \left(
t\right) -\psi \left( s\right) \right) ^{\alpha -1}K\left( t,s,f\left(
s\right) \right) ds%
\end{array}%
\right\Vert _{C_{1-\gamma ;\psi }} \\
&\leq &\left\Vert g\left( t,f\left( t\right) \right) -g\left( t,h\left(
t\right) \right) \right\Vert _{C_{1-\gamma ;\psi }} \\
&&+\frac{1}{\Gamma \left( \alpha \right) }\displaystyle\int_{0}^{t}\psi
^{\prime }\left( s\right) \left( \psi \left( t\right) -\psi \left( s\right)
\right) ^{\alpha -1}\left\Vert K\left( t,s,f\left( s\right) \right) -K\left(
t,s,h\left( s\right) \right) \right\Vert _{C_{1-\gamma ;\psi }}ds \\
&\leq &L_{1}\left\Vert f\left( t\right) -h\left( t\right) \right\Vert
_{C_{1-\gamma ;\psi }} \\
&&+\frac{L_{2}}{\Gamma \left( \alpha \right) }\displaystyle\int_{0}^{t}\psi
^{\prime }\left( s\right) \left( \psi \left( t\right) -\psi \left( s\right)
\right) ^{\alpha -1}\left\Vert f\left( s\right) -h\left( s\right)
\right\Vert _{C_{1-\gamma ;\psi }}ds \\
&\leq &L_{1}C_{fh}\phi \left( t\right) +L_{2}C_{fh}\phi \left( t\right) \leq \left( L_{1}+L_{2}\right) d\left( f,h\right) \phi \left( t\right)  \\
&\leq &\frac{\left( L_{1}+L_{2}\right) }{1-\left( L_{1}+L_{2}\right) L}\phi
\left( t\right) 
\end{eqnarray*}
which implies that $d\left( \Lambda f,h\right) <\infty $.
\end{proof}

\section{Application}

We recall that for a Banach's space $\mathcal{X} $, a one parameter family $%
\left\{U_{\alpha }\left( t\right) \right\} _{t\geq 0}$, $\alpha >0$ in $%
B\left(\mathcal{X} \right) ,$ the space of all bounded linear operators, is
called an $\alpha-$resolvent operator function if the following conditions
are satisfied \cite{chen}:

\begin{enumerate}
\item $U_{\alpha }\left( t\right) $ is strongly continuous on $\mathbb{R}
^{+}$ and $U_{\alpha }\left( 0\right) =I$;

\item $U_{\alpha }\left( s\right) U_{\alpha }\left( t\right) =U_{\alpha
}\left( t\right) U_{\alpha }\left( s\right) ,$ for all $s,t\geq 0$;

\item The functional equation 
\begin{equation*}
U_{\alpha }\left( s\right) I_{0+,t}^{\alpha ;\psi }U_{\alpha }\left(
t\right) -I_{0+,s}^{\alpha ;\psi }U_{\alpha }\left( s\right) U_{\alpha
}\left( t\right) =I_{0+,t}^{\alpha ;\psi }U_{\alpha }\left( t\right)
-I_{0+,s}^{\alpha ;\psi }U_{\alpha }\left( t\right)
\end{equation*}
holds for all $s,t\geq 0$.
\end{enumerate}

The generator $A$ of $U_{\alpha }$ is defined by 
\begin{equation*}
D\left( A\right) :=\left\{ x\in \mathcal{X} :\underset{t\rightarrow 0^{+}}{%
\lim }\frac{U_{\alpha }\left( t\right) x-x}{g_{\alpha +1}\left( t\right) }%
\text{ exists}\right\}
\end{equation*}
where $g_{\alpha +1}\left( t\right) =\dfrac{t^{\alpha }}{\Gamma \left(\alpha
+1\right) }$ and 
\begin{equation*}
Ax:=\underset{t\rightarrow 0^{+}}{\lim }\frac{U_{\alpha }\left( t\right) x-x 
}{g_{\alpha +1}\left( t\right) };\text{ }x\in D\left( A\right).
\end{equation*}

Note that, $\underset{t\rightarrow 0^{+}}{\lim }U_{\alpha }\left( t\right)
x=x$ for $x\in \mathcal{X}$.

\begin{exemplo}
Let $\mathcal{X}$ be a Banach's space, $A\in B\left( \mathcal{X}\right) $
with $\left\Vert A\right\Vert \leq 1,$ $T\in \left( 0,\infty \right) $ and $%
\Phi \left( \cdot \right) \in W^{1,1}\left( \mathbb{R}_{+},\mathbb{C} \right)
$, where $W^{1,1}$ is the Sobolev space. From \textrm{Corollary \ref{cor1}},
we know that there exists a unique solution $u\in C^{1}\left( \mathbb{R},%
\mathcal{X} \right) \cap C\left( \mathbb{R},\mathcal{X} \right) $ satisfying
the fractional Volterra integro-differential equation 
\begin{equation}  \label{eq34}
^{H}\mathbb{D}_{0+}^{\alpha ,\beta ;\psi }u\left( t\right) =Au\left(
t\right)+f\left( t\right) +\int_{0}^{t}\Phi \left( t-s\right) Au\left(
s\right) ds.
\end{equation}

First, we define $g:\left[ 0,T\right] \times \mathcal{X} \rightarrow 
\mathcal{X} $, and $K:\left[0,T\right] \times \left[ 0,T\right] \times 
\mathcal{X} \rightarrow \mathcal{X} $ by $g\left( t,x\right) =Ax+f\left(
t\right) $ and $K\left( t,s,x\right) =\Phi\left( t-s\right) Ax$.

Note that, 
\begin{equation*}
\left\Vert g\left( t,x\right) -g\left( t,y\right) \right\Vert _{C_{1-\gamma
;\psi }}=\left\Vert Ax-Ay\right\Vert _{C_{1-\gamma ;\psi }}\leq \left\Vert
A\right\Vert \left\Vert x-y\right\Vert _{C_{1-\gamma ;\psi }}
\end{equation*}
and on the other hand 
\begin{eqnarray*}
\left\Vert K\left( t,s,x\right) -K\left( t,s,y\right) \right\Vert
_{C_{1-\gamma ;\psi }} &=&\left\vert \Phi \left( t-s\right) \right\vert
\left\Vert Ax-Ay\right\Vert _{C_{1-\gamma ;\psi }}  \notag \\
&\leq &M\left\Vert A\right\Vert \left\Vert x-y\right\Vert _{C_{1-\gamma
;\psi }}
\end{eqnarray*}
for some $M>0$. Suppose 
\begin{equation*}
0<L<\frac{1-\left\Vert A\right\Vert }{\left\Vert A\right\Vert \left(
1+M+MT\right) },
\end{equation*}
$\alpha \geq \dfrac{1}{L}$ and $\phi \left( t\right) =\rho $ $e^{\alpha t},$ 
$\rho >0$. If 
\begin{equation*}
\left\Vert ^{H}\mathbb{D}_{0+}^{\alpha ,\beta ;\psi }u\left( t\right)
-Au\left(t\right) -f\left( t\right) -\int_{0}^{t}\Phi \left( t-s\right)
Au\left(s\right) ds\right\Vert _{C_{1-\gamma ;\psi }}\leq \phi \left(
t\right)
\end{equation*}
for an appropriate $f$ and $\Phi \left( \cdot \right) $, then with $%
L_{1}=\left\Vert A\right\Vert $ and $L_{2}=M\left\Vert A\right\Vert$, by 
\textnormal{\ Theorem \ref{teo2}}, there exists a unique solution $%
u_{0}\left( t\right) $ of \textnormal{Eq.(\ref{eq34})} such that 
\begin{eqnarray*}
&&\left\Vert ^{H}\mathbb{D}_{0+}^{\alpha ,\beta ;\psi }u\left( t\right) -%
\text{ }^{H}\mathbb{D}_{0+}^{\alpha ,\beta ;\psi }u_{0}\left( t\right)
\right\Vert_{C_{1-\gamma ;\psi }}+\left\Vert u\left( t\right) -u_{0}\left(
t\right)\right\Vert _{C_{1-\gamma ;\psi }}  \notag \\
&\leq &\frac{1+L}{1-L_{1}+\left( L_{1}+L_{2}+L_{2}L\right) L}\phi \left(
t\right).
\end{eqnarray*}

Thus we obtain the Ulam-Hyers stability of \textnormal{\ Eq.(\ref{eq34})}.
\end{exemplo}

\section{Concluding remarks}

The study of the stability of solutions of differential equations is
intensifying with the years and several researchers have presented new and
interesting results involving fractional derivatives. In this paper, we
presented an investigation of the Ulam-Hyers stability of the solution of a
fractional Volterra integro-differential equation by means of the Banach
fixed-point theorem. Besides that, we introduced the concept of $\alpha$%
-resolvent and presented, as an example, Ulam-Hyers stability in the Sobolev
space.

Also, the study of the stability of Ulam-Hyers and Ulam-Hyers-Rassias is
indeed interesting and is not restricted only to fractional differential
equations, in particular of the type given by Eq.(\ref{eq1}) and also to the
particular $\psi-$Hilfer fractional derivative. We can use another type of
derivative, for example, conformable derivative, which constitutes a change
of scale in relation to the integer-order derivative \cite{ZE3}, as well as
the fractional derivative with non-singular kernel, as proposed by
Caputo-Fabrizio \cite{ZE4}. With this, we can propose another study
involving stabilities of the type Ulam-Hyers, Ulam-Hyers-Rassias,
Ulam-Hyers-Bourgin, Aoki-Rassias, among others associated with the solutions
of the functional differential equations \cite{theory}.

\bibliography{ref}
\bibliographystyle{plain}

\end{document}